\documentclass[11pt]{article}
\usepackage[margin=1in]{geometry}

\title{The Stembridge Equality for Skew Stable Grothendieck Polynomials and Skew Dual Stable Grothendieck Polynomials}
\author{Fiona Abney-McPeek, Serena An, and Jakin Ng}
\date{\today}

\usepackage[utf8]{inputenc}
\usepackage{graphicx}
\usepackage{amsmath, amsfonts, amssymb, amsthm, amscd}
\usepackage{textcomp}
\usepackage{arcs}
\usepackage{epstopdf}
\usepackage{enumitem}
\usepackage{epic}
\usepackage{setspace}
\usepackage{ytableau}
\usepackage{tikz}
\usepackage{caption}
\usepackage{mathtools}
\usepackage{dsfont}
\usetikzlibrary{matrix,arrows,decorations.pathmorphing}
\usepackage[colorinlistoftodos]{todonotes}
\usepackage[colorlinks=true, allcolors=blue]{hyperref}

\theoremstyle{definition}
\newtheorem{definition}{Definition}[section]
\newtheorem{example}[definition]{Example}
\newtheorem{remark}[definition]{Remark}

\theoremstyle{plain} 
\newtheorem{theorem}[definition]{Theorem}
\newtheorem{lemma}[definition]{Lemma}
\newtheorem{corollary}[definition]{Corollary}

\newtheorem{proposition}[definition]{Proposition}

\DeclareMathOperator{\Par}{Par}

\makeatletter
\renewcommand*\env@matrix[1][\arraystretch]{%
  \edef\arraystretch{#1}%
  \hskip -\arraycolsep
  \let\@ifnextchar\new@ifnextchar
  \array{*\c@MaxMatrixCols c}}
\makeatother

\begin{document}
\maketitle

\begin{abstract}
    
    The Schur polynomials $s_{\lambda}$ are essential in understanding the representation theory of the general linear group. They also describe the cohomology ring of the Grassmannians. For $\rho = (n, n-1, \dots, 1)$ a staircase shape and $\mu \subseteq \rho$ a subpartition, the Stembridge equality states that $s_{\rho/\mu} = s_{\rho/\mu^T}$. This equality provides information about the symmetry of the cohomology ring. The stable Grothendieck polynomials $G_{\lambda}$, and the dual stable Grothendieck polynomials $g_{\lambda}$, developed by Buch, Lam, and Pylyavskyy, are variants of the Schur polynomials and describe the $K$-theory of the Grassmannians. Using the Hopf algebra structure of the ring of symmetric functions and a generalized Littlewood-Richardson rule, we prove that $G_{\rho/\mu} = G_{\rho/\mu^T}$ and $g_{\rho/\mu} = g_{\rho/\mu^T}$, the analogues of the Stembridge equality for the skew stable and skew dual stable Grothendieck polynomials.

\end{abstract}

\section{Introduction}

In this paper, we prove a Stembridge-type equality for skew stable Grothendieck polynomials and skew dual stable Grothendieck polynomials, namely
\[
G_{\rho/\mu} = G_{\rho/\mu^T},\qquad g_{\rho/\mu} = g_{\rho/\mu^T},
\] where $\rho = (n, n-1, \dots, 1)$ is the staircase partition. 

The \textit{stable Grothendieck polynomials} $G_\lambda$ are $K$-theoretic analogues of the Schur polynomials $s_{\lambda}$, i.e., they provide information about the $K$-theory of the Grassmanian. These formal power series were introduced by Fomin and Kirillov \cite{fomin}. 

In \cite{buch}, Buch gave a combinatorial definition of the skew stable Grothendieck polynomials $G_{\lambda/\mu}$ using set-valued tableaux of shape $\lambda/\mu$, which are certain fillings of the skew Young diagram of the shape $\lambda/\mu$ with sets of positive integers.
The \textit{dual stable Grothendieck polynomials} $g_{\lambda}$, first introduced by Lam and Pylyavskyy in \cite{LP}, are dual to the $G_{\lambda}$'s under the Hall inner product. The formal power series $g_{\lambda/\mu}$ is defined using reverse plane partitions, which are certain fillings of the skew Young diagram of the shape $\lambda/\mu$ with positive integers.

The skew stable and skew dual stable Grothendieck polynomials can be viewed as deformations of the Schur polynomials in that their lowest and highest degree parts, respectively, are the Schur polynomials. Thus, it is natural to ask whether certain identities for the Schur polynomials can be extended to this context. For instance, it was conjectured in \cite[Conjecture 6.2]{involve} that there are analogues for $g_{\rho/\mu}$ and $G_{\rho/\mu}$ of the Stembridge equality \cite[Corollary 7.32]{reiner}, which states that
\[
s_{\rho/\mu} = s_{\rho/\mu^T},
\] for $\rho = (n, n-1, \dots, 1)$. 
In this paper, we verify that the conjectures hold. More precisely, we prove that there is a Stembridge-type equality for the skew stable and skew dual stable Grothendieck polynomials, thereby exhibiting additional symmetries on the $K$-theory of the Grassmanians.

\subsection{Outline of the paper}

In Section 2, we begin by going over the basics of symmetric functions, Schur polynomials, and skew stable and skew dual stable Grothendieck polynomials. Then, we state the problem and introduce the Hopf algebraic structure of the ring of symmetric functions $\Lambda$. 

In Section 3.1, we give a combinatorial proof of Lemma \ref{straight shape}, the Stembridge equality for skew dual stable Grothendieck polynomials $g_{\rho/\mu}$ in the special case where $\mu = (k)$, using a generalized Littlewood-Richardson rule for the stable Grothendieck polynomials proven by Buch \cite{buch}.
\begin{lemma}\label{straight shape}
Let $\rho=(n,n-1,\ldots,1)$ be the staircase partition, and $\mu = (k)$ where $k \le n$. Then,
\begin{align*}
    g_{\rho/\mu} = g_{\rho/{\mu}^T}.
\end{align*}
\end{lemma}

In Section 3.2, we prove the Stembridge equality for skew dual stable Grothendieck polynomials for general $\mu$, as stated in the following theorem.
\begin{theorem}\label{main result small g}
Let $\rho=(n,n-1,\ldots,1)$ be the staircase partition, and $\mu\subseteq\rho$ any subpartition. Then $$g_{\rho/\mu}=g_{\rho/{\mu}^T}.$$ 
\end{theorem}
We extend Lemma \ref{straight shape}, the case for $\mu = (k)$, to Theorem \ref{main result small g}, the general case, by utilizing the skewing operator $\perp$ coming from the Hopf algebraic structure of $\Lambda$, along with an involution $\tau$ of the completion $\hat{\Lambda}$ constructed by Yelliusizov in \cite[Theorem 1.1]{duality and deformations} sending $G_\mu$ to $G_{\mu^T}$.

In Section 4, we use a similar strategy to prove the Stembridge equality for skew stable Grothendieck polynomials $G_{\rho/\mu}$.
\begin{theorem}\label{main result big g}
Let $\rho=(n,n-1,\ldots,1)$ be the staircase partition, and $\mu\subseteq\rho$ any subpartition. Then $$G_{\rho/\mu}=G_{\rho/{\mu}^T}.$$ 
\end{theorem}
First, we prove the identity combinatorially for the case $\mu = (k)$ in Section 4.1, and then generalize it to arbitrary $\mu$ in Section 4.2 using the skewing operator and an involution $\overline{\tau}$ of $\Lambda$, an analogue of $\tau$ sending $g_\mu$ to $g_{\mu^T}$, introduced by Yelliusizov \cite{duality and deformations}.

\begin{remark}

In Section 2.8, we prove that the converses of Theorem \ref{main result small g} and Theorem  \ref{main result big g} are true. That is, if $G_{\rho/\mu}=G_{\rho/{\mu}^T}$ (respectively, $g_{\rho/\mu}=g_{\rho/{\mu}^T}$) for all $\mu\subseteq \rho$, then $\rho = (n, n-1, \dots, 1)$ for some nonnegative integer $n$. This follows from Corollary \ref{converse of stembridge}, the converse statement in the case of Schur polynomials, since $s_{\lambda/\mu}$ is the bottom degree component of $G_{\lambda/\mu}$ and the top degree component of $g_{\lambda/\mu}$.
\end{remark}

\section{Preliminaries}
\subsection{Partitions and Diagrams}
A \emph{partition} $\lambda$ of a nonnegative integer $n$ is a weakly decreasing sequence of positive integers $(\lambda_1, \lambda_2, \dots, \lambda_{\ell})$ whose sum is $n$. The integer $\lambda_i$ is the $i$th \emph{part} of $\lambda$. The number of parts of $\lambda$ is the \emph{length} of $\lambda$, denoted $\ell(\lambda)$. We define $|\lambda| = \lambda_1 + \lambda_2 + \cdots + \lambda_{\ell}$. Denote the set of all partitions of $n$ by $\Par(n)$, and let $\Par := \bigcup_{n\ge 0} \Par(n)$. 

\begin{definition}
    The \emph{Young diagram} of a partition $\lambda$, denoted $Y(\lambda)$, is a left-aligned array with $\lambda_i$ cells in the $i$th row from the top.
\end{definition}

\pagebreak

For example, 
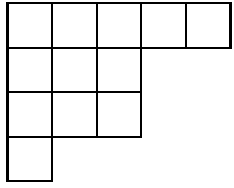
\begin{center}
    \ydiagram{5, 3, 3, 1}
\end{center}
is the Young diagram of $\lambda = (5, 3, 3, 1)$.

If $\lambda$ and $\mu$ are two partitions such that $\mu_i \le \lambda_i$ for all $i$, then we write $\mu\subseteq \lambda$ and say that $\mu$ is a \emph{subpartition} of $\lambda$. We may additionally consider the \emph{skew partition} $\lambda/\mu$ whose \emph{skew Young diagram} consists of the cells belonging to $Y(\lambda)$ but not to $Y(\mu)$. 
For example, 
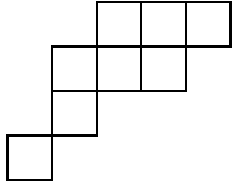
\begin{center}
    \ydiagram{2 + 3, 1 + 3, 1 + 1, 1}
\end{center}
is the Young diagram of $\lambda/\mu = (5, 4, 2, 1)/(2, 1, 1)$. For a skew partition $\lambda/\mu$, we define $|\lambda/\mu| = |\lambda| - |\mu|$. We also identify the partition $\lambda$ with the skew partition $\lambda/\emptyset$.

The \emph{conjugate} of a partition $\lambda$, denoted $\lambda^T$, is the partition whose $i$th part is the number of entries of $\lambda$ that are at least $i$. Equivalently, $Y(\lambda^T)$ is obtained from $Y(\lambda)$ by a reflection over the main diagonal. For example, $(4, 2, 1)$ and $(3, 2, 1, 1)$ are conjugates, as seen from their Young diagrams below.
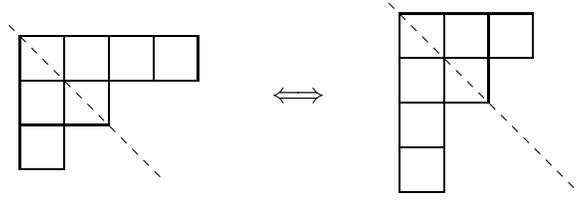
\[
\begin{matrix}
        \begin{tikzpicture}[baseline=(current bounding box.center)]\node (n) {    \ydiagram{4, 2, 1}};\draw[dashed] (n.north west) -- ([xshift=-0.6cm]n.south east);\end{tikzpicture}
& \quad \iff \quad  &
\begin{tikzpicture}[baseline=(current bounding box.center)]
\node (m) {
    \ydiagram{3, 2, 1, 1}
   };
\draw[dashed] (m.north west) -- ([xshift=0.6cm]m.south east);
\end{tikzpicture}

\end{matrix}
\]

\begin{definition}
Let $\rho_n$ denote the staircase partition $(n, n-1, \dots, 1)$ for some $n\ge 1$. We may use $\rho$ (omitting the $n$) to denote a general staircase partition of unspecified size.
\end{definition}

\subsection{Symmetric Functions}
A \textit{weak composition} of $n$, for $n\in \mathbb{N}$, is an infinite sequence of nonegative integers $\alpha = (\alpha_1, \alpha_2, \dots)$ with $\sum \alpha_i = n.$ Define $x^{\alpha} \coloneqq x_1^{\alpha_1}x_2^{\alpha_2}\cdots$. 
A \textit{homogeneous symmetric function of degree $n$} is a formal power series
\[f(x) = \sum_{\alpha}c_{\alpha}x^{\alpha}\] 
such that $\alpha$ ranges over all weak composition of $n$, the $c_{\alpha}$ are elements of some commutative ring $R$, and for each permutation $\omega$ of the positive integers, $f(x_1, x_2, \dots) = f(x_{\omega(1)}, x_{\omega(2)}, \dots)$. For our purposes, we will take $R = \mathbb{Q}$ and let $\Lambda^n$ denote the set of all homogeneous symmetric functions of degree $n$ over $\mathbb{Q}$. Additionally, $\Lambda = \Lambda^0 \oplus \Lambda^1 \oplus \cdots$, the set of all symmetric functions, is a graded algebra over $\mathbb{Q}$. 

\begin{definition}
    The \textit{elementary symmetric function} $e_n$ is given by $$e_n \coloneqq \sum_{i_1 < \cdots < i_n} x_{i_1}\cdots x_{i_n}.$$ For a partition $\lambda = (\lambda_1, \lambda_2, \dots),$ let $$e_{\lambda} \coloneqq e_{\lambda_1}e_{\lambda_2} \cdots.$$  
\end{definition}
The set $\{e_{\lambda}\}$ for all partitions $\lambda$ forms a basis for $\Lambda.$
\begin{definition}
    The \textit{complete homogeneous symmetric function} $h_n$ is given by $$h_n \coloneqq \sum_{i_1 \leq \cdots \leq i_n} x_{i_1}\cdots x_{i_n}.$$ In particular, $h_n$ is the sum of all monomials with degree $n$. For a partition $\lambda = (\lambda_1, \lambda_2, \dots)$, let \[h_{\lambda} = h_{\lambda_1}h_{\lambda_2}\cdots .\]
\end{definition}
The set $\{h_{\lambda}\}$ for all partitions $\lambda$ forms a basis for $\Lambda$.

\subsection{Schur Polynomials}
\begin{definition}
    A \emph{semistandard Young tableau} (SSYT) of shape $\lambda/\mu$ is a filling of the cells of $Y(\lambda/\mu)$ with positive integers such that the entries weakly increase within each row and strictly increase within each column.
    A semistandard Young tableau $T$ has \emph{type} $\alpha = (\alpha_1, \alpha_2, \dots)$ where $\alpha_i$ is the number of entries of $T$ equal to $i$.
\end{definition}
 
For example, 
\begin{center}
    \ytableausetup{notabloids}
    \begin{ytableau}
    \none & \none & \none & 2 & 4\\
    \none & 1 & 1 & 4\\
    1 & 2 & 2\\
    3 & 4\\
    6
\end{ytableau}
\end{center}
is a SSYT of shape $(5, 4, 3, 2, 1)/(3, 1)$ and type $(3, 3, 1, 3, 0, 1)$.

For a SSYT $T$ of type $\alpha = (\alpha_1, \alpha_2, \dots)$, let $x^T$ denote $x_1^{\alpha_1} x_2^{\alpha_2}\cdots$.

\begin{definition}
    For a skew shape $\lambda/\mu$, the \emph{skew Schur polynomial} $s_{\lambda/\mu}$ in the variables $x = (x_1, x_2, \dots)$ is given by
    \[s_{\lambda/\mu} = \sum_{T} x^T,\]
    where the sum is over all SSYT $T$ of shape $\lambda/\mu$.
    When $\mu = \emptyset$, then $s_{\lambda}$ is the \emph{Schur polynomial} of $\lambda$.
\end{definition}

\begin{example}
    Every SSYT $T$ of shape $\lambda/\mu = (2, 1, 1) /(1)$ is of one of the following forms for some positive integers $i<j<k$.
    \[
    \ytableaushort{\none i, i, j}\quad
    \ytableaushort{\none j, i, j}\quad
    \ytableaushort{\none i, j, k}\quad
    \ytableaushort{\none j, i, k}\quad
    \ytableaushort{\none k, i, j}
    \ytableausetup{nosmalltableaux}
    \]
    Thus, 
    \begin{align*}
        s_{(2, 1, 1)/(1)} &= \sum_{i<j}x_i^2x_j + \sum_{i<j}x_ix_j^2 + 3\sum_{i<j<k}x_ix_jx_k.\\
    \end{align*}
\end{example}

Next, we state two well-known properties of Schur polynomials (see \cite[Chapter 7]{stanley}).
\begin{theorem}
    For all skew partitions $\lambda/\mu$, the skew Schur polynomial $s_{\lambda/\mu}$ is a symmetric function.
\end{theorem}

\begin{theorem}
    The set $\{s_{\lambda} : \lambda\in\Par(n)\}$ forms a basis for $\Lambda^n$, and the set $\{s_{\lambda} : \lambda\in\Par\}$ forms a basis for $\Lambda$.
\end{theorem}

\begin{definition}
    
The \emph{Hall inner product} $\langle \cdot , \cdot \rangle$ on $\Lambda$ is defined so that the Schur polynomials are orthonormal; that is, $\langle s_{\lambda}, s_{\mu} \rangle = \delta_{\lambda \mu}$, the Kronecker delta. 
\end{definition}

\subsection{Dual Stable Grothendieck Polynomials}
\begin{definition}
    A \emph{reverse plane partition} of shape $\lambda/\mu$ is a filling of the cells of $Y(\lambda/\mu)$ with positive integers such that the entries weakly increase within each row and column. A reverse plane partition $P$ has \emph{weight} $w = (w_1, w_2, \dots)$, where $w_i$ is the number of columns of $P$ containing $i$.
\end{definition}

For example,
\begin{center}
    \begin{ytableau}
    \none & 1 & 2 & 2 & 4\\
    \none & 1 & 2 & 5\\
    1 & 2 & 2
\end{ytableau}
\end{center}
is a reverse plane partition of shape $(5, 4, 3)/(1, 1)$ and weight $(2, 3, 0, 1, 1)$.

For a reverse plane partition $P$ of weight $w = (w_1, w_2, \dots)$, let $x^P$ denote $x_1^{w_1}x_2^{w_2}\cdots$.

\begin{definition}
For a skew shape $\lambda / \mu$, define the \emph{skew dual stable Grothendieck polynomial} $g_{\lambda / \mu}$ to be
\begin{align*}
    g_{\lambda/\mu} = \sum_{P}x^P,
\end{align*}
where the sum is over all reverse plane partitions $P$ of shape $\lambda / \mu$. 
When $\mu = \emptyset$, then $g_{\lambda}$ is the \emph{dual stable Grothendieck polynomial} of $\lambda$.
\end{definition}

\begin{example}
Every reverse plane partition $P$ of shape $\lambda/\mu = (2, 2)/(1)$ takes on one of the following forms, for some positive integers $i<j<k$.
    \[
    \ytableaushort{\none i, i i}\quad
    \ytableaushort{\none i, i j}\quad
    \ytableaushort{\none j, i j}\quad
    \ytableaushort{\none i, j j}\quad
    \ytableaushort{\none i, j k}\quad
    \ytableaushort{\none j, i k}
    \]
Thus,
\begin{align*}
    g_{(2, 2)/(1)} &= \sum_{i}x_i^2 + \sum_{i<j}x_i^2 x_j + \sum_{i<j}x_ix_j + \sum_{i<j}x_ix_j^2 + 2\sum_{i<j<k}x_ix_jx_k.
\end{align*}
\end{example}

As shown in \cite{LP}, the dual stable Grothendieck polynomials $g_{\lambda}$ are symmetric functions and form a basis for $\Lambda$.

\begin{remark}\label{s is highest}
The terms of highest degree in $g_{\lambda/\mu}$ are achieved by reverse plane partitions in which there are no numbers repeated in any column; that is, the columns are strictly increasing. In other words, the reverse plane partition must also be a semi-standard Young tableau. Thus, the terms of highest degree in $g_{\lambda/\mu}$ form $s_{\lambda/\mu}$.
\end{remark}

\subsection{Stable Grothendieck Polynomials}
\begin{definition}
    For two nonempty sets $A$ and $B$ of positive integers, we say that $A\le B$ if $\max A\le \min B$ and $A < B$ if $\max A < \min B$.
    A \textit{set-valued tableau} of shape $\lambda/\mu$ is then a filling of the boxes of $Y(\lambda/\mu)$ with nonempty sets of positive integers such that the sets weakly increase along rows and strictly increase along columns.
\end{definition}
\begin{definition}
Let the \textit{size} of $T$, denoted by $|T|$, be the sum of the sizes of the sets appearing in $T$. 
\end{definition}
\begin{example}
The following is a set-valued tableau of shape $(5, 4, 3)/(2, 1)$ and size $15.$
    \begin{center}
    \ytableausetup{notabloids, boxsize = 3em}
    \begin{ytableau}
    \none & \none & 1,2 & 2,3,4 & 7\\
    \none & 3 & 3,5 & 5\\
    2 & 4,5,6 & 6
    \end{ytableau}
\end{center}
\end{example}
\begin{definition}
Let $m_i$ be the number of times that $i$ appears in the set-valued tableau $T$, and let $x^T = x_1^{m_1}x_2^{m_2}\cdots$. Then the \textit{skew stable Grothendieck polynomial} $G_{\lambda/\mu}$ is a formal power series given by
\[
G_{\lambda/\mu} \coloneqq \sum_{T}(-1)^{|T|-|\lambda/\mu|}x^T,
\]
where the sum is over all set-valued tableaux $T$ of shape $\lambda/\mu.$ When $\mu = \emptyset$, then $G_{\lambda}$ is the \emph{stable Grothendieck polynomial} of $\lambda$.
\end{definition}

\begin{remark}\label{s is lowest}
A set-valued tableau of shape $\lambda/\mu$ filled with sets of size one is a semi-standard Young tableau, corresponding to the monomials in $G_{\lambda/\mu}$ of lowest degree.
Thus, the terms of lowest degree in $G_{\lambda/\mu}$ form $s_{\lambda/\mu}.$ The stable Grothendieck polynomial has terms of arbitrarily large degree if $|\lambda/\mu| > 0$.
\end{remark}

\begin{remark}
Let $\hat{\Lambda}$ be the completion of $\Lambda$, given by allowing infinite linear combinations of a given basis (e.g. the Schur polynomials). The Hall inner product $\langle \cdot, \cdot \rangle: \Lambda \times \Lambda \rightarrow \mathbb{Q}$ can be extended to a pairing $$\langle \cdot, \cdot \rangle: \hat{\Lambda}\times\Lambda\rightarrow \mathbb{Q}$$ by linearly extending $\langle s_{\lambda}, s_{\mu} \rangle = \delta_{\lambda \mu}$ as in \cite{duality and deformations}. The $G_{\lambda}$ are symmetric functions, and any symmetric formal power series $f\in\hat{\Lambda}$ can be \textit{uniquely} represented as an infinite sum $\sum_{\lambda\in\Par}\alpha_{\lambda}G_{\lambda}$ with $a_{\lambda}\in\mathbb{Q}$.
The $G_{\lambda}$ are also dual to the $g_{\lambda}$ under the (extended) Hall inner product; that is, $\langle G_{\lambda}, g_{\mu} \rangle = \delta_{\lambda\mu}$. 
\end{remark}

\subsection{Hopf Algebras}
The ring of symmetric functions $\Lambda$ has a Hopf algebraic structure, as described in \cite{hopf}. To compute $\Delta(f)$ for a symmetric function $f\in\Lambda$, we introduce new indeterminates $y_1, y_2, \dots$ and write the power series $f(x_1, x_2, \dots, y_1, y_2, \dots)$ as a finite sum
\[f(x_1, x_2, \dots, y_1, y_2, \dots) = \sum_{i = 1}^{k} p_i(x_1, x_2, \dots)q_i(y_1, y_2, \dots),\]
for $p_i, q_i\in \Lambda$. Then
\[\Delta(f) = \sum_{i = 1}^{k} p_i\otimes q_i.\]

\begin{lemma}\label{skewing thingy of g lambda}
    The comultiplication acts on $g_{\lambda}$ as follows:  
    
\[\Delta(g_{\lambda}) = \sum_{\mu \subseteq \lambda} g_{\mu} \otimes g_{\lambda/\mu }.\]
\end{lemma}

\begin{proof}
    From the combinatorial definition, 
    \[
    g_{\lambda}(x_1, x_2, \dots, y_1, y_2, \dots) = \sum_{P} x^P,
    \]
    summed over all reverse plane partitions $P$ of shape $\lambda$ with entries in the alphabet $x_1 < x_2 < \cdots < y_1 < y_2 < \cdots$. Since the rows and columns of $P$ are weakly increasing, the restriction of the reverse plane partition to the alphabet $x$ gives a reverse plane partition $P_x$ of shape $\mu \subseteq \lambda$, and the restriction to the alphabet $y$ gives a reverse plane partition $P_y$ of shape $\lambda / \mu$. 
    
    Then, 
    
    \[
    g_{\lambda}(x, y) = \sum_P x^{P_x} \cdot y^{P_y} = \sum_{\mu \subseteq \lambda} \left(\sum_{P_x} x^{P_x}\right)\left(\sum_{P_y} y^{P_y}\right) = \sum_{\mu \subseteq \lambda} g_{\mu}(x) g_{\lambda/\mu}(y).
    \]
    
    Thus, we indeed have that 
    
    \[\Delta(g_{\lambda}) = \sum_{\mu \subseteq \lambda} g_{\mu} \otimes g_{\lambda/\mu }.\]

\vspace*{-2\baselineskip}
\end{proof}

We next define the skewing operator $\perp$ (see for instance \cite{hopf}, Section 2.8) which we will use throughout the paper.

\begin{definition}\label{skewing}
Let $f\in \Lambda$ or $f\in \hat{\Lambda}$. The \textit{skewing operator} $f^{\perp} : \Lambda\to\Lambda$ is defined by
\[f^{\perp}(a)\coloneqq\sum_{i = 1}^{k}\langle f, b_i\rangle c_i,\]
where $\Delta(a)$ is written as $\Delta(a) = \sum_{i = 1}^{k}b_i\otimes c_i$.

\end{definition}

\begin{theorem}\label{skewing little g} For any partition $\mu \subseteq \lambda$,
    \[G_{\mu}^{\perp}g_{\lambda} = g_{\lambda/\mu}.\]
\end{theorem}

\begin{proof}
    Recall from Lemma \ref{skewing thingy of g lambda} that $\Delta(g_{\lambda}) = \sum_{\mu \subseteq \lambda} g_{\mu} \otimes g_{\lambda/\mu }.$ 
    Then by the definition of the skewing operator, 
    \begin{align*}
        G_{\mu}^{\perp} g_{\lambda} &= \sum_{\nu} \langle G_{\mu}, g_{\nu}\rangle g_{\lambda /\nu} \\
    &= \sum_{\nu} \delta_{\mu, \nu} g_{\lambda/\nu} \\
    &= g_{\lambda/ \mu}.
    \end{align*}
    
    \vspace*{-2\baselineskip}
\end{proof}

The skewing operator also has the following useful properties. 

\begin{lemma}
[]
    \label{skewingProperties1}
    For $f, g \in \hat{\Lambda}$ and $a \in \Lambda,$ we have  
    \begin{align*}\langle g, f^{\perp}(a)\rangle = \langle fg,a\rangle ,\end{align*}
    where $\langle \cdot, \cdot \rangle$ denotes the extended Hall inner product. 
    
\end{lemma}
\begin{proof}
For $f = s_{\lambda}, g = s_{\nu} \in \Lambda,$ the equation reduces to Proposition 2.8.2 of \cite{hopf}, which gives us that $\langle s_{\nu}, s_{\lambda}^{\perp}(a) \rangle = \langle s_{\lambda}s_{\nu}, a\rangle.$ Now, for any $f = \sum_{\lambda}a_{\lambda} s_{\lambda}, g = \sum_{\nu}b_{\nu}s_{\nu} \in \hat{\Lambda},$ the skewing operator distributes linearly, and the inner product is bilinear. Since $a$ is a finite linear combination of Schur polynomials, the terms of large degree in $f$ and $g$ do not contribute to the final sum, since the inner product of a higher-degree Schur polynomial with $a$ will be zero. As a result, both inner products are sums of finitely many terms. Thus since the Schur polynomials form a basis for $\hat{\Lambda}$, the equation holds true for general $f, g.$
\end{proof}

\begin{lemma}\label{skewingProperties2}
    For $f, g \in \Lambda$ and $a \in \hat{\Lambda},$ we have  
    \begin{align*}\langle g, f^{\perp}(a)\rangle = \langle fg,a\rangle .\end{align*}
\end{lemma}
The proof for this lemma follows in a way analogous to that of Lemma \ref{skewingProperties1}, since the terms of large degree in $a$ can now be ignored. Here the Hall inner product is linearly extended in the second coordinate. 

\subsection{The Stembridge Equality}
The Stembridge equality describes an important symmetry for the Schur polynomials and can be proved in a number of different ways (e.g. Corollary 7.32 in \cite{reiner} and Exercise 2.9.25 in \cite{hopf}). 

\begin{theorem}[Stembridge Equality]
    Let $\rho=(n,n-1,\dots,1)$ be the staircase partition, and $\mu \subseteq \rho$. Then    $$s_{\rho/\mu} = s_{\rho/\mu^T}.$$
\end{theorem}

In this paper, we extend the Stembridge equality to the skew stable and skew dual stable Grothendieck polynomials.

In addition, the converse is true. That is, if $s_{\lambda/\mu} = s_{\lambda/\mu^T}$ for all $\mu\subseteq\lambda$, then $\lambda = (n, n-1, \dots, 1)$ for some nonnegative integer $n$. To prove this, we use Pieri's rule, a well-known fact described in \cite[Theorem 7.15.7]{stanley}, for example. 

\begin{definition}
    A skew shape $\lambda/\nu$ is a \textit{horizontal strip} if it has no two squares in the same column, or a \textit{vertical strip} if no two squares are in the same row.
\end{definition}

\begin{theorem}\label{pieri}[Pieri's Rule]
    We have 
    \[s_{\lambda/(k)} = \sum_{\nu}s_{\nu},\]
    where $\nu$ ranges over all partitions $\nu\subseteq\lambda$ for which $\lambda/\nu$ is a horizontal strip of size $k$. Similarly,
    \[s_{\lambda/(1^k)} = \sum_{\nu}s_{\nu},\]
    where $\nu$ ranges over all partitions $\nu\subseteq\lambda$ for which $\lambda/\nu$ is a vertical strip of size $k$.
\end{theorem}

\begin{theorem}
    If $s_{\lambda/(k)} = s_{\lambda/(1^k)}$ for all nonnegative integers $k$, then $\lambda = \rho_n = (n, n-1, \dots, 1)$ for some nonnegative integer $n$.
\end{theorem}

\begin{proof}
    Note that $s_{\lambda/(k)}$ is zero if and only if $k$ is greater than the number of columns in the Young diagram of $\lambda,$ and $s_{\lambda/(1^k)}$ is zero if and only if $k$ is greater than the number of rows in the Young diagram of $\lambda.$ So we require that $Y(\lambda)$ has the same number of rows as columns; let this number be $n$.
    
    We then require that $\lambda/\nu$ referenced in Theorem \ref{pieri} (Pieri’s Rule) is a horizontal strip of size $k$ if and only if it is a vertical strip of size $k$, since the $s_{\nu}$ form a basis of $\Lambda$.
    
    For the sake of contradiction, suppose that $\lambda$ contains two consecutive parts of the same size. Then there exists some $\nu$ such that $\lambda/\nu$ consists of the rightmost box of these two rows. However, then $\lambda/\nu$ forms a vertical strip of length 2 but not a horizontal strip, which is a contradiction. 
    
    Combining the fact that $\lambda$ has $n$ rows and $n$ columns and that no two rows have the same size, we have that $\lambda$ must be $\rho_n = (n, n-1, \dots, 1)$ for some nonnegative integer $n$, as desired.
\end{proof}

\begin{corollary}\label{converse of stembridge}
    If for some partition $\lambda,$ $s_{\lambda/\mu} = s_{\lambda/\mu^T}$ for all partitions $\mu$, then $\lambda = \rho_n = (n, n-1, \dots, 1)$ for some nonnegative integer $n$.
\end{corollary}

In Section 2.8, we will extend this converse to the skew stable and skew dual stable Grothendieck polynomials.

\subsection{Statement of the Problem}
Now, we are ready to introduce our first main result, Theorem \ref{main result small g}, an analogue of the Stembridge equality for the dual stable Grothendieck polynomials, which states that for $\rho = (n, n-1, \dots, 1),$
\begin{align*}
    g_{\rho/\mu} = g_{\rho/\mu^T}.
\end{align*}
We first prove a special case of this theorem, Lemma \ref{straight shape}, for when $\mu$ is the partition $(k)$ or $(k)^T = (1^k),$ in Section 3.1 using a bijection between set-valued tableaux. Then we extend this to general $\mu$ using the stable Grothendieck polynomials and Hopf algebraic structure of the symmetric functions in Section 3.2.

Our second main result, Theorem \ref{main result big g}, is an analogue of the Stembridge equality for the stable Grothendieck polynomials, stating that
\begin{align*}
    G_{\rho/\mu} = G_{\rho/\mu^T}.
\end{align*}
Similar to the dual stable Grothendieck polynomial case, we will first prove this theorem for $\mu = (k)$ or $(1^k)$ in Section 4.1 by finding a bijection between set-valued tableaux. Then, we use the Hopf algebraic structure to extend to general $\mu$ in Section 4.2.

\begin{example}
Consider $\rho = (3, 2, 1)$ and $\mu = (2)$. The diagrams for $\rho/\mu$ and $\rho/\mu^T$ are below.

 \begin{center}
    \ytableausetup{notabloids, boxsize = 2em}
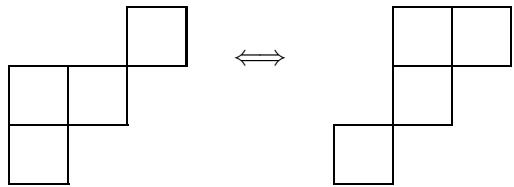
    \begin{ytableau}
    \none & \none & \\
     & \\
    
    \end{ytableau} \quad $\iff$ \quad \begin{ytableau}
    \none &  & \\
    \none & \\
    \color{white} . 
    \end{ytableau}
\end{center}

For $\rho/\mu$, the top right section does not share any columns with the rest of the diagram, so the number occupying the top right square is unconstrained by the remainder of the diagram. Then 
$$g_{\rho / \mu} = g_{21}\cdot g_1,$$ the product of the two symmetric functions.
The same argument holds for $\rho / \mu^T,$ since the bottom left section is independent of the top right section, and so
$$g_{\rho/\mu^T} = g_1\cdot g_{21},$$ and the two polynomials are equal. 
\end{example}

We can also prove the converses of our main results (Theorems \ref{main result big g} and \ref{main result small g}) by extending Corollary \ref{converse of stembridge}.
    
    \begin{theorem}\label{converse for g's}
      If $g_{\rho/\mu} = g_{\rho/\mu^T}$ or $G_{\rho/\mu} = G_{\rho/\mu^T},$ for all $\mu,$ then $\rho = (n, n-1, \dots, 1)$ for some nonnegative integer $n$.
      
    \end{theorem}
    \begin{proof}
        As stated in Remarks \ref{s is highest} and \ref{s is lowest}, the equalities $g_{\rho/\mu} = g_{\rho/\mu^T}$ and $G_{\rho/\mu} = G_{\rho/\mu^T}$ both require the Stembridge equality $s_{\rho/\mu} = s_{\rho/\mu^T}$, which in turn requires $\rho = (n, n-1, \dots, 1)$ for some nonnegative integer $n$.
    \end{proof}

\section{Proof for Dual Stable Grothendieck Polynomials}
In this section we prove Theorem \ref{main result small g}, by first proving with the special case when $\mu = (k)$ or $(1^k)$ combinatorially, and then generalizing to arbitrary $\mu$ using the Hopf algebraic structure of the symmetric functions.

\subsection{Proof for $\mu = (k)$ or $(1^k)$}
Our proof for Lemma \ref{straight shape} makes use of the skewing operator $\perp,$ described in Section 2.6, and Theorem \ref{buch} by Buch \cite{buch}.

\begin{definition}
    Let $w(T)$ denote the \textit{reverse reading word} of a set-valued tableau $T$, read top to bottom along a column, starting with the rightmost column and moving left, and with the elements within a box read largest to smallest.\footnote{Here our definition of $w(T)$ is the reverse of the $w(T)$ as defined in \cite{buch}.} 

\end{definition}
For example, the following set-valued tableau has a reverse reading word of 743252153636542.
    \begin{center}
    \ytableausetup{notabloids, boxsize = 3em}
    \begin{ytableau}
    \none & \none & 1,2 & 2,3,4 & 7\\
    \none & 3 & 3,5 & 5\\
    2 & 4,5,6 & 6
    \end{ytableau}
\end{center}
    
\begin{definition}
    A reverse reading word is a \textit{lattice word} if the $i$th instance of $a+1$ comes after the $i$th instance of $a$ for all positive integers $i$ and $a$. The \textit{content} of a word is $(w_1, w_2, \dots)$ where $w_i$ is the number of times that $i$ appears in the word.
\end{definition}

For example, 1121322 is a lattice word, but 121221 is not.

\begin{definition}
    Let $\nu\ast\mu$ denote the skew shape formed by joining the partitions $\nu$ and $\mu$ such that the top right corner of $\mu$ touches the bottom left corner of $\nu$.\footnote{Our $\nu \ast \mu$ is the $\mu \ast \nu$ of \cite{buch}.} 
\end{definition} 

For example, we have $(2, 1) \ast (4) = (6, 5, 4) /(4, 4).$
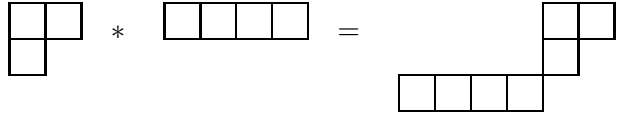
\begin{center}
\begin{center}
    \ytableausetup{notabloids, boxsize = 1.2em}

    \ydiagram{2, 1}\quad  $\ast$ \quad \ydiagram{4}  \quad = \quad \ydiagram{4 + 2, 4 + 1, 4}
\end{center}
\end{center}

Next, we have the following theorem, a Littlewood-Richardson rule for stable Grothendieck polynomials, as shown by Buch (\cite{buch}, Theorem 5.4).

\begin{theorem}[Buch]\label{buch}

    Let $\nu$ and $\mu$ be two partitions. Then,
\[G_{\nu}G_{\mu} = \sum_{\lambda}(-1)^{|\lambda| - |\nu| - |\mu|}c_{\nu\mu}^{\lambda}G_{\lambda},\]
where $c_{\nu\mu}^{\lambda}$ is the number of set-valued tableaux $T$ of shape $\mu\ast\nu$ such that $w(T)$ is a lattice word with content $\lambda$.\footnote{Here our $c_{\nu\mu}^{\lambda}$ is in fact $(-1)^{|\lambda| - |\nu| - |\mu|}c_{\nu\mu}^{\lambda}$ as defined in \cite{buch}. Our Theorem 3.4 as stated here is equivalent to Theorem 5.4 in \cite{buch}, as our $w(T)$ is the reverse and our $\nu \ast \mu$ is flipped.} 
\end{theorem}

A \textit{valid filling} of a set-valued tableau $T$ is a filling such that $w(T)$ is a lattice word with content $\rho_n=(n, n-1,\dots,1)$ for some $n$. 

\begin{lemma}\label{contents of nu}
    In a valid filling of $\nu\ast\mu$, all boxes in the $i$th row of $\nu$ contain the set $\{i\}$.
\end{lemma} 

\begin{proof}
    The rightmost box in the first row of $\nu$ must contain the set $\{1\}$, because a lattice word must begin with 1, so all boxes in the first row contain the set $\{1\}$, as rows are increasing left to right. The rightmost box in the second row may only contain numbers greater than 1, and in order for $\nu$'s reading word to be a lattice word, this box must contain the set $\{2\}$. Thus, all boxes of the second row contain the set $\{2\}$. Analogously, we may inductively show that all boxes in the $i$th row of $\nu$ contain the set $\{i\}$.
\end{proof}

If $w(T) = \rho_n,$ the integer $j$ appears in the filling $n-j + 1$ times. 
Since the columns in a set-valued tableau are strictly increasing, there is at most one $j$ in each column of $\nu\ast\mu$. In particular, in a valid filling of $\nu\ast (1^k)$, $(1^k)$ contains at most one of each of the numbers $1, 2, \dots, n$. In fact, the same holds for a valid filling of $\nu\ast (k)$.

\begin{lemma}\label{contents of k}
    In a valid filling of $\nu\ast (k)$ with content $\rho_n$, the filling of $(k)$ contains at most one of each of the numbers $1, 2, \dots, n$.
\end{lemma}
\begin{proof}
    For the sake of contradiction, suppose that the filling of the shape $(k)$ contains at least two $i$'s. Note that $i\le n-1$ because the valid filling with content $\rho_n$ contains only one $n$. Then in the reverse reading word of the filling, all $i+1$'s are listed before the second-to-last $i$. In other words, the $(n-i)$th $i+1$ is listed before the $(n-i)$th $i$, contradicting the assumption that the reverse reading word of the filling of $\nu\ast (k)$ is a lattice word.
\end{proof}

\begin{theorem}\label{straightshapec's}
    We have $c^{\rho}_{(k)\nu} = c^{\rho}_{(1^k)\nu}$ for all positive integers $k\le n$.
\end{theorem}
\begin{proof}

Considering Lemma \ref{contents of nu} and Lemma \ref{contents of k}, there is at most one $i$ in the filling of $(k)$ (resp. $(1^k)$) in a valid filling of $\nu\ast (k)$ (resp. $\nu\ast (1^k)$) with content $\rho_n$. This means that filling of the $i$th row of $\nu$ contains exactly $n-i+1$ or $n-i$ $i$'s. 
If $i+1$ is in the filling $(k)$ or $(1^k)$, it must be the $(n-i)$th occurrence of the value $i+1$ in the valid filling of $\nu\ast (k)$ or $\nu\ast (1^k)$, respectively. Since there are at least $n-i$ occurrences of the value $i$ found in $\nu$, any arrangement of the numbers in $(k)$ concatenated after the reverse reading word of $\nu$ will form a lattice word.

So given a valid filling of $\nu\ast (k)$, we can obtain a corresponding valid filling of $\nu\ast (1^k)$ upon rotating $(k)$ by 90 degrees clockwise. Similarly, given a valid filling of $\nu\ast (1^k)$, we can obtain a corresponding valid filling of $\nu\ast (k)$ upon rotating $(1^k)$ by 90 degrees counterclockwise. Therefore, there is a bijection between valid fillings of $\nu\ast (k)$ and $\nu\ast (1^k)$, and $c^{\rho}_{(k)\nu} = c^{\rho}_{(1^k)\nu}$.
\end{proof}

For example, the following are corresponding set-valued tableaux under this bijection, for $\rho = (5, 4, 3, 2, 1)$, $\nu = (4, 4, 2, 1)$, and $(k) = (3)$.

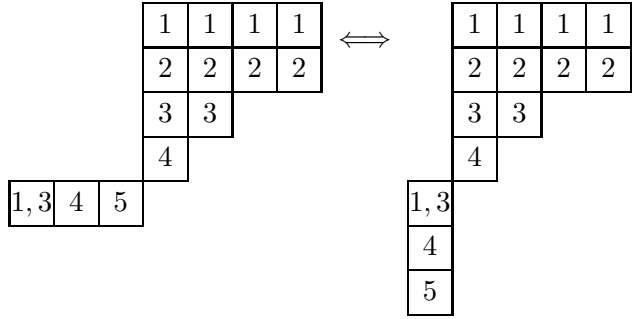
\ytableausetup{boxsize = 1.5em}
    \[\begin{ytableau}
    \none & \none & \none & 1 & 1 & 1 & 1\\ 
    \none & \none & \none & 2 & 2 & 2 & 2\\ 
    \none & \none & \none & 3 & 3\\ 
    \none & \none & \none & 4\\ 
    1, 3 & 4 & 5
    \end{ytableau}
    \iff
    \begin{ytableau}
    \none & 1 & 1 & 1 & 1\\ 
    \none & 2 & 2 & 2 & 2\\ 
    \none & 3 & 3\\ 
    \none & 4\\ 
    1, 3 \\
    4 \\
    5
    \end{ytableau}
    \]

\begin{lemma}\label{c's to grothendieck}
Fix $\rho = (n, n-1, \dots, 1)$ and some partition $\mu \subseteq \rho$. If $c^{\rho}_{\mu\nu} = c^{\rho}_{\mu^T\nu}$ for all partitions $\nu$, then $g_{\rho/\mu} = g_{\rho/\mu^T}$.
\end{lemma}
\begin{proof}
Assume that $c^{\rho}_{\mu\nu} = c^{\rho}_{\mu^T\nu}$ for all partitions $\nu$. 
Recall from Theorem \ref{buch} that
\[G_{\nu}G_{\mu} = \sum_{\rho}(-1)^{|\rho| - |\nu| - |\mu|}c_{\mu\nu}^{\rho}G_{\rho},\]
so 
\begin{align*}
    \langle G_{\nu}G_{\mu}, g_{\rho}\rangle &= (-1)^{|\rho|-|\nu|-|\mu|}c_{\mu\nu}^{\rho}\\
    &= (-1)^{|\rho|-|\nu|-|\mu^t|}c_{\mu^T\nu}^{\rho}\\
    &= \langle G_{\nu}G_{\mu^T}, g_{\rho}\rangle.
\end{align*}

Combining with Lemma \ref{skewingProperties1} and Theorem \ref{skewing little g}, we have 
\begin{align*}
    \langle G_{\nu}G_{\mu}, g_{\rho}\rangle &= \langle G_{\nu}G_{\mu^T}, g_{\rho}\rangle\\
    \implies\langle G_{\nu}, G_{\mu}^{\perp}g_{\rho} \rangle &= \langle G_{\nu}, G_{\mu^T}^{\perp}g_{\rho} \rangle\\
    \implies \langle G_{\nu}, g_{\rho/\mu} \rangle &= \langle G_{\nu}, g_{\rho/\mu^T} \rangle
\end{align*}
for all partitions $\nu$.
Since the $G_{\nu}$ form a basis for $\hat{\Lambda}$, we conclude that $g_{\rho/\mu}  = g_{\rho/\mu^T}.$
\end{proof}

Now, we give the proof of Lemma $\ref{straight shape}$.
\begin{proof}
    By Theorem \ref{straightshapec's}, we have that $c^{\rho}_{\nu (k)} = c^{\rho}_{\nu (1^k)}$. Then Lemma \ref{c's to grothendieck} gives $g_{\rho/(k)} = g_{\rho/(1^k)}$.
\end{proof}

\subsection{Proof for All Partitions}

Recall from Section 2.5 that $\hat{\Lambda}$ is the completion of $\Lambda$, the ring of symmetric functions. 

We take the linear map $\tau: \hat{\Lambda} \to \hat{\Lambda}$, given by continuously and linearly extending $G_{\lambda} \mapsto G_{\lambda^T}$; see for instance \cite{duality and deformations} or Corollary 6.7 of \cite{buch}. This map is a continuous ring endomorphism and an involution of $\hat{\Lambda}$.

We will begin by proving the following theorem.

\begin{theorem}\label{true for all e_k}
For $\rho = (n, n-1, \dots, 1)$ and all positive integers $k$, we have
    \[e_k^{\perp}g_{\rho} = \tau(e_k)^{\perp}g_{\rho}.\]
\end{theorem}

In order to do so, we shall first prove the following lemmas and propositions.

\begin{lemma}\label{straight shape G}
The stable Grothendieck polynomial of shape $(1^k)$ can be written as 
\begin{align*}
G_{(1^k)} = \sum_{n \ge k}(-1)^{n-k}\binom{n-1}{k-1}e_n.
\end{align*}
\end{lemma}
\begin{proof}
The stable Grothendieck polynomial $G_{(1^k)}$ is a sum over set-valued tableaux $T$ of shape $(1^k)$. All set-valued tableaux are strictly increasing along the columns, so all entries are distinct in a set-valued tableau of shape $(1^k)$. Therefore, the monomial $x^T$ for each set-valued tableau $T$ is of the form $x_{i_1}x_{i_2}\cdots x_{i_n}$ for positive integers $i_1<\cdots < i_n,$ where $n = |T|$. Given $n$ distinct entries, the number of ways to fill in $T$ is equal to the number of compositions of $n$ into $k$ nonempty parts, which is $\binom{n-1}{k-1}.$

Therefore, 
\begin{align*}
G_{(1^k)} &= \sum_T (-1)^{|T|-|1^k|} x^T \\
&= \sum_n \sum_{|T| = n} (-1)^{n-k} x^T \\
&= \sum_{n \geq k} (-1)^{n-k} \binom{n-1}{k-1} \sum_{i_1 < \cdots < i_n} x_{i_1}\cdots x_{i_n} \\
&= \sum_{n \geq k} (-1)^{n-k} \binom{n-1}{k-1} e_n,
\end{align*}
as desired.
\end{proof}

\begin{proposition}\label{e as linear combo of G}
We can write the elementary symmetric function $e_k$ as an infinite sum of stable Grothendieck polynomials:
\begin{align*}
    e_k = \sum_{n \geq k} \dbinom{n-1}{k-1} G_{(1^n)}.
\end{align*}
\end{proposition}
\begin{proof}
By Lemma \ref{straight shape G},
\begin{align}\label{uses trinomial revision}
    \sum_{n \geq k} \dbinom{n-1}{k-1} G_{(1^n)} &= \sum_{k \leq n} \sum_{n \leq j} (-1)^{j-n} \dbinom{n-1}{k-1} \dbinom{j-1}{n-1} e_j. 
\end{align}

For a given $j \ge k,$ the coefficient of $e_j$ on the right hand side of (\ref{uses trinomial revision}) is
\begin{align*}
     \sum_{k \le n \le j} (-1)^{j-n} \dbinom{n-1}{k-1} \dbinom{j-1}{n-1} &=\sum_{k \le n \le j} (-1)^{j-n} \dbinom{j-1}{k-1} \dbinom{j-k}{n-k} \\
     &= \dbinom{j-1}{k-1} \delta_{kj} = \delta_{kj},
\end{align*}
where the first simplification comes from trinomial revision and the second comes from the fact that the alternating sum of a row of binomial coefficients (besides the first row) is $0$. This means that the coefficient of $e_j$ in (\ref{uses trinomial revision}) is 0 for all $j \ne k$ and 1 for $j = k$, so 
\begin{align*}
    e_k = \sum_{n \geq k} \dbinom{n-1}{k-1} G_{(1^n)}.
\end{align*}

\vspace*{-2\baselineskip}
\end{proof}

\begin{lemma}\label{perp distributes over infinite sums}
The skewing operator $\perp$ distributes over well-defined infinite sums; that is, for any $a\in \Lambda$ and $f_1, f_2, \cdots \in \hat{\Lambda}$,
\begin{align*}
    \Big(\sum_{i\ge 1}f_i\Big)^{\perp}(a) = \sum_{i \ge 1}f_i^{\perp}(a).
\end{align*}
\end{lemma}
\begin{proof}
Let $\Delta(a) = \sum_{(a)} a_{(1)}\otimes a_{(2)}$. Then 
\begin{align*}
    \Big(\sum_{i\ge 1}f_i\Big)^{\perp}(a) = \sum_{(a)} \Big\langle\sum_{i\ge 1}  f_i, a_{(1)} \Big\rangle a_{(2)}.
\end{align*}
Since the Hall inner product on $\hat{\Lambda} \times \Lambda \mapsto \mathbb{Z}$ is bilinear and continuous in its first argument, we expand to get
\begin{align*}
\Big(\sum_{i\ge 1}f_i\Big)^{\perp}(a) &= \sum_{(a)} \sum_{i\ge 1}\langle f_i, a_{(1)} \rangle a_{(2)}  \\
&= \sum_{i \ge 1} \sum_{(a)}\langle f_i, a_{(1)} \rangle a_{(2)} \\
&= \sum_{i \ge 1}f_i^{\perp}(a).
\end{align*}

\vspace*{-2\baselineskip}
\end{proof}

Finally, we can give the proof of Theorem \ref{true for all e_k}.

\begin{proof}
 Observe that $\tau$ is a continuous ring endomorphism of $\hat{\Lambda}$, as is the skewing operator by Lemma \ref{perp distributes over infinite sums}. Applying Proposition \ref{e as linear combo of G} to decompose $e_k$ as a linear combination of $G_{(1^n)}$, and applying Lemma \ref{straight shape} 
to write $G_{(1^n)}^{\perp}(g_{\rho}) = G_{(n)}^{\perp}(g_{\rho})$ gives:
\begin{align*}
    \tau(e_k)^{\perp}(g_{\rho}) &= \tau \Big( \sum_{n \ge k} \dbinom{n-1}{k-1} G_{(1^n)}\Big)^{\perp}(g_{\rho})\\
    &= \sum_{n \ge k} \dbinom{n-1}{k-1} \tau(G_{(1^n)})^{\perp}(g_{\rho}) \\
    &= \sum_{n \ge k} \dbinom{n-1}{k-1} G_{(n)}^{\perp}(g_{\rho}) \\
    &= \sum_{n \ge k} \dbinom{n-1}{k-1} G_{(1^n)}^{\perp}(g_{\rho}) \\
    &=  \Big( \sum_{n \ge k} \dbinom{n-1}{k-1} G_{(1^n)}\Big)^{\perp}(g_{\rho}) \\
    &= e_k^{\perp}g_{\rho}.
    \end{align*}
    
    \vspace*{-2\baselineskip}
    \end{proof}
    
Now, in order to prove Theorem \ref{main result small g}, we need the following lemmas, which are inspired by Exercises 2.9.24 and 2.9.25 in \cite{hopf}.

\begin{lemma}\label{subalgebra for arbitrary homomorphism}
Let $\psi$ be an arbitrary continuous ring endomorphism of $\hat{\Lambda}$. Then for any given $a \in \Lambda$, the set $A = \{f \in \hat{\Lambda} : f^{\perp}(a) = \psi(f)^{\perp} (a)\}$ is closed under finite multiplication and (possibly infinite) addition of its elements.
\end{lemma}

\begin{proof}
First, we show that if $f_1, f_2,\cdots\in A$, then $\sum_{i \ge 1} f_i\in A$. By considering Lemma \ref{perp distributes over infinite sums} and using the fact that $\psi$ is a continuous ring endomorphism of $\hat{\Lambda}$, we have
\begin{align*}
    \Big( \psi \Big(\sum_{i \ge 1} f_i\Big) \Big)^{\perp} (a) &= \Big(\sum_{i \ge 1}\psi(f_i)\Big)^{\perp} (a)\\
    &=\sum_{i \ge 1}\psi(f_i)^{\perp}(a)\\
    &= \sum_{i \ge 1}f_i^{\perp} (a)\\
    &= \Big(\sum_{i \ge 1}f_i\Big)^{\perp}(a),
\end{align*}
so indeed, $\sum_{i \ge 1} f_i \in A$.

Next, we show that if $f_1, f_2 \in A,$ then $f_1f_2 \in A.$ 
First, notice that if $\langle k, f_1 \rangle = \langle k, f_2 \rangle$ for all $k \in \Lambda$, then $f_1 = f_2$. This is because each $f_1$ and $f_2$ can be written uniquely as a (possibly infinite) linear combination of Schur polynomials $s_{\lambda}$, and taking $k$ to be $s_{\lambda}$ for all $\lambda$ in turn, gives us that $f_1$ and $f_2$ have the same coefficient for all  $s_{\lambda}$. Thus, it suffices to show that $\langle k, (f_1f_2)^{\perp}(a) \rangle = \langle k, (\psi(f_1f_2))^{\perp}(a) \rangle$ for all $k \in \Lambda$.

Repeatedly applying the property from Lemma \ref{skewingProperties1}, and using the fact that $f_1, f_2 \in A$, we can do the following manipulation:
\begin{align*}
\langle k, (f_1f_2)^{\perp}(a)\rangle &= \langle f_1f_2k, a \rangle \\
&= \langle f_2k, f_1^{\perp}(a) \rangle \\
&= \langle f_2k, {\psi}(f_1)^{\perp}(a) \rangle \\
&= \langle {\psi}(f_1)f_2k, a \rangle \\
&= \langle f_2{\psi}(f_1)k, a \rangle \\
&= \langle {\psi}(f_1)k, f_2^{\perp}(a) \rangle \\
&= \langle {\psi}(f_1)k, {\psi}(f_2)^{\perp}(a) \rangle \\
&= \langle {\psi}(f_2){\psi}(f_1)k, a \rangle \\
&= \langle {\psi}(f_1f_2)k, a \rangle \\
&= \langle k, {\psi}(f_1f_2)^{\perp}(a) \rangle.
\end{align*}

This means that $(f_1f_2)^{\perp}(a) = {\psi}(f_1f_2)^{\perp}(a),$ so indeed $f_1f_2 \in A$.
Thus $A$ is closed under finite multiplication and (possibly infinite) addition.
\end{proof}

\begin{corollary}\label{set is algebra for tau}
The set $A = \{f \in \hat{\Lambda} : f^{\perp}(g_{\rho}) = \tau(f)^{\perp} (g_{\rho})\}$, where $\tau$ is the ring homomorphism defined earlier by linearly extending $G_{\lambda} \mapsto G_{\lambda^T},$ is closed under finite multiplication and infinite addition.
\end{corollary}

Now we are ready to prove Theorem \ref{main result small g}.

\begin{proof}
By Theorem \ref{true for all e_k}, we have $e_k^{\perp}g_{\rho} = \tau(e_k)^{\perp}g_{\rho}$, so this means $e_k \in A = \{f \in \hat{\Lambda} : f^{\perp}(g_{\rho}) = \tau(f)^{\perp} (g_{\rho})\}$. By Corollary \ref{set is algebra for tau}, $A$ is closed under multiplication and possibly infinite sums, so for all $\lambda$ we have $e_{\lambda} = e_{\lambda_1}e_{\lambda_2} \cdots e_{\lambda_n} \in A$. Since the $e_{\lambda}$ form a basis for $\hat{\Lambda}$ (if we allow infinite linear combinations), any symmetric function $f = \sum_{\lambda}a_{\lambda}e_{\lambda}$ is in $A$. In particular, $G_{\mu} \in A$, so
\begin{align*}
    g_{\rho/\mu} &= G_{\mu}^{\perp}g_{\rho} \\
        &= \tau (G_{\mu})^{\perp}g_{\rho} \\
        &= G_{\mu^T}^{\perp} g_{\rho} \\
        &= g_{\rho/\mu^T}.
\end{align*}

\vspace*{-2\baselineskip}
\end{proof}

\section{Proof for Stable Grothendieck Polynomials}
In this section we prove Theorem \ref{main result big g}, by first proving with the special case when $\mu = (k)$ or $(1^k)$ combinatorially, and then generalizing to arbitrary $\mu$ using the Hopf algebraic structure of the symmetric functions.

\subsection{Proof for $\mu = (k)$ or $(1^k)$}
Throughout this section, we denote the partition $(n, n-1, \dots, 1)$ by $\rho_n$. The following theorem by Buch (\cite{buch}, Theorem 6.9) allows us to prove the special case when $\mu = (k)$ or $(1^k)$ combinatorially.

\begin{theorem}[Buch]\label{big G buch}
For a skew partition $\lambda/\mu$,
\[G_{\lambda / \mu} = \sum_{\nu} (-1)^{|\nu| - |\lambda/\mu|}\alpha_{\lambda/\mu, \nu} G_{\nu},\]
where the coefficient $\alpha_{\lambda/\mu, \nu}$ is the number of set-valued tableaux $T$ of shape $\lambda/\mu$ such that $w(T)$ is a lattice word with content $\nu$.
\end{theorem}

Now, we have the following recurrence for the $\alpha_{\lambda/\mu,\nu}$ coefficients when $\lambda = \rho_n$ and $\mu = (k)$ for some positive integer $k$. 
\begin{lemma}
Fix a partition $\nu = (\nu_1, \nu_2, \dots, \nu_m)$, and let $\nu^{-} = (\nu_2, \dots, \nu_m)$. For  given positive integers $k,n$ with $k<n$, we have
\[\alpha_{\rho_n / (k), \nu} = \alpha_{\rho_{n-1} / (k), \nu^{-}} + 2\alpha_{\rho_{n-1} / (k-1), \nu^{-}}.\]
\end{lemma} 

\begin{proof} 
Consider a set-valued tableau $T$ of shape $\rho_n/(k)$ such that $w(T)$ is a lattice word with content $\nu$. 
The rightmost box in the first row of $T$ must contain the set $\{1\}$, so all boxes in the first row contain the set $\{1\}$. The rightmost box in the second row must contain the set $\{2\}$, and similarly the rightmost $n-k-1$ boxes in the second row must all contain the set $\{2\}$. 

\begin{center}
    \ytableausetup{notabloids}
    \begin{ytableau}
    \none & \none & \none & 1 & 1 & 1\\
     &  & B & 2 & 2\\
     &  &  & \\
     &  & \\
     & \\
     
\end{ytableau}
\end{center}

Since rows are weakly increasing, there are three cases for the contents of the $k$th box $B$ in the second row: $\{1\}$, $\{2\}$, or $\{1, 2\}$. 

Let the function $f$ map the tableau $T$ to the tableau $T^{-}$ by deleting all 1's from the boxes of $T$, deleting all now-empty boxes, and subtracting 1 from all remaining numbers. Assuming that $w(T)$ is a lattice word, $w(T^{-})$ is also a lattice word: the $i$th $a+1$ coming before the $i$th $a+2$ in $w(T)$ corresponds to the $i$th $a$, which comes before the $i$th $a+1$ in $w(T^{-})$.

\bigskip

\noindent\textbf{Case 1: $B$ contains $\{1\}$}

Then, the leftmost $k$ boxes in the second row must all contain the set $\{1\}$, since the rows are weakly increasing. 

\begin{center}
    \ytableausetup{notabloids}
    \begin{ytableau}
    \none & \none & \none & 1 & 1 & 1\\
    1 & 1 & 1 & 2 & 2\\
     &  &  & \\
     &  & \\
     & \\
     
\end{ytableau}
\end{center}

Applying $f$ gives a tableau $T^{-}$ of shape  $\rho_{n-1}/(k)$ and content $\nu^{-}$, as illustrated below. 

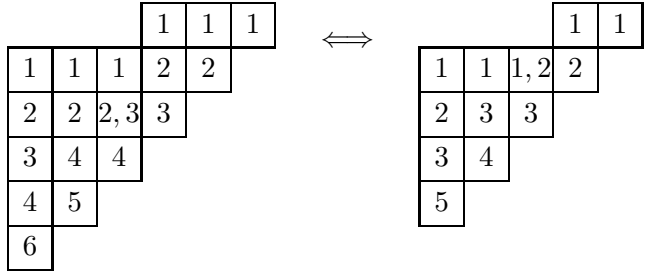
\[
\ytableausetup{notabloids}
    \begin{ytableau}
    \none & \none & \none & 1 & 1 & 1 \\
    1 & 1 & 1 & 2 & 2\\
    2 & 2 & 2, 3 & 3\\
    3 & 4 & 4\\
    4 & 5\\
    6
\end{ytableau} 
\quad \iff \quad 
    \begin{ytableau}
    \none &  \none & \none & 1 & 1\\
    1 & 1 & 1, 2 & 2\\
    2 & 3 & 3\\
    3 & 4\\
    5
\end{ytableau}\]
Notice that if a tableau $T^{-}$ of shape $\rho_{n-1}/(k-1)$ has lattice reverse reading word, then it necessarily corresponds to exactly one tableau $T$ of shape $\rho_n/(k)$ with lattice reverse reading word where $B$ contains $\{1\}$. Then these two sets are in bijection, so there are $\alpha_{\rho_{n-1}/(k), \nu^{-}}$ tableaux $T$ of shape $\rho_n/(k)$ with lattice reverse reading word where $B$ contains $\{1\}$.

\bigskip

\textbf{Case 2: $B$ contains $\{2\}$}

Then the $(k-1)$th leftmost box in the second row must contain $\{1\}$, as otherwise the $(n-k+1)$th 2 would come before the $(n-k+1)$th 1 in $w(T)$ and it would not be a lattice word.

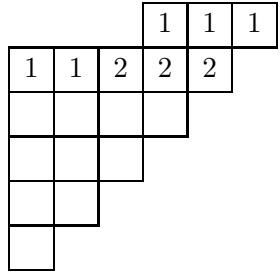
\begin{center}
    \ytableausetup{notabloids}
    \begin{ytableau}
    \none & \none & \none & 1 & 1 & 1\\
    1 & 1 & 2 & 2 & 2\\
     &  &  & \\
     &  & \\
     & \\
     
\end{ytableau}
\end{center}

Applying $f$ gives a tableau $T^{-}$ of shape $\rho_{n-1}/(k-1)$ with content $\nu^{-}$, as shown below.
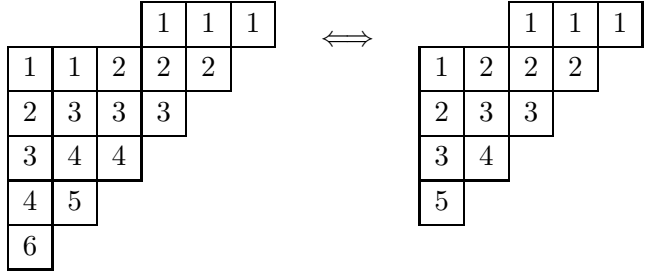
\[
    \ytableausetup{notabloids}
    \begin{ytableau}
    \none & \none & \none & 1 & 1 & 1\\
    1 & 1 & 2 & 2 & 2\\
    2 & 3 & 3 & 3\\
    3 & 4 & 4\\
    4 & 5\\
    6 
    \end{ytableau} \quad \iff \quad
    \begin{ytableau}
    \none & \none & 1 & 1 & 1\\
    1 & 2 & 2 & 2\\
    2 & 3 & 3\\
    3 & 4\\
    5
    \end{ytableau} 
\]

Notice that as in case 1, $f$ is a bijection here, so there are $\alpha_{\rho_{n-1}/(k-1), \nu^{-}}$ tableaux $T$ such that $B$ contains $\{2\}.$

\bigskip

\textbf{Case 3: $B$ contains $\{1, 2\}$}

Then the leftmost $k-1$ boxes in the second row must all contain the set $\{1\}$. 
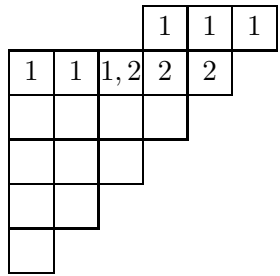
\begin{center}
    \ytableausetup{notabloids}
    \begin{ytableau}
    \none & \none & \none & 1 & 1 & 1\\
    1 & 1 & 1, 2 & 2 & 2\\
     &  &  & \\
     &  & \\
     & \\
     
\end{ytableau}
\end{center}

The map $f$ gives a tableau of shape $\rho_{n-1}/(k-1)$ with content $\nu^{-}$ which has reverse column word a lattice word. Note that in this case, the inverse map will additionally insert an extra 1 in the $k$th box of the second row, as illustrated below.

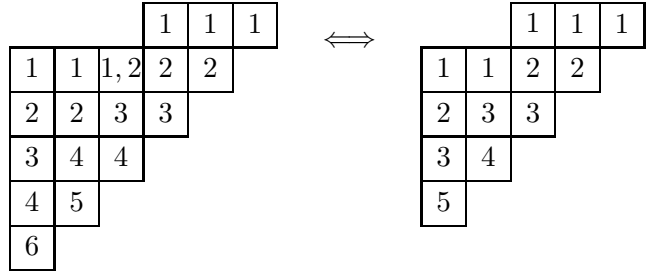
\[
\ytableausetup{notabloids}
    \begin{ytableau}
    \none & \none & \none & 1 & 1 & 1 \\
    1 & 1 & 1, 2 & 2 & 2\\
    2 & 2 & 3 & 3\\
    3 & 4 & 4\\
    4 & 5\\
    6
\end{ytableau} \quad \iff \quad 
\ytableausetup{notabloids}
    \begin{ytableau}
    \none &  \none & 1 & 1 & 1\\
    1 & 1 & 2 & 2\\
    2 & 3 & 3\\
    3 & 4\\
    5
\end{ytableau} 
\]

As above, we have a bijection, and so there are $\alpha_{\rho_{n-1}/(k-1), \nu^{-}}$ tableaux $T$ such that $B$ contains $\{1, 2\}.$

Combining all three cases, we have
\[\alpha_{\rho_n / (k), \nu} = \alpha_{\rho_{n-1} / (k), \nu^{-}} + 2\alpha_{\rho_{n-1} / (k-1), \nu^{-}}.\]

\vspace*{-2\baselineskip}
\end{proof}

There exists a similar recurrence for the $\alpha_{\lambda/\mu,\nu}$ coefficients when $\lambda = \rho_n$ and $\mu = (1^k)$ for some positive integer $k$. 
\begin{lemma}
Fix a partition $\nu = (\nu_1, \nu_2, \dots, \nu_m)$, and let $\nu^{-} = (\nu_2, \dots, \nu_m)$. For a given positive integer $n$, we have
\[\alpha_{\rho_n / (1^k), \nu} = \alpha_{\rho_{n-1} / (1^k), \nu^{-}} + 2\alpha_{\rho_{n-1} / (1^{k-1}), \nu^{-}}.\]
\end{lemma}

\begin{proof}
Consider a set-valued tableau $T$ of shape $\rho_{n}/(1^k)$ such that $w(T)$ is a lattice word with content $\nu$. The rightmost box in the first row of $T$ must contain the set $\{1\}$, so all boxes in the first row contain the set $\{1\}$. The rightmost box in the second row must contain the set $\{2\}$, so all boxes in the second row must contain the set $\{2\}$. Analogously, all boxes in the $j$th row for $1\le j\le k$ must contain the set $\{j\}$. 

There are three possibilities for the leftmost box $B$ in the $k+1$th row: $B$ contains $\{1\}$, a set without a 1, or a set containing 1 of size at least 2, as illustrated below. Define $f$ as in the previous lemma. 

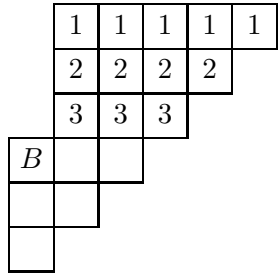
\begin{center}
    \ytableausetup{notabloids}
    \begin{ytableau}
    \none & 1 & 1 & 1 & 1 & 1\\
    \none & 2 & 2 & 2 & 2\\
    \none & 3 & 3 & 3\\
    B &  & \\
     & \\
     
\end{ytableau}
\end{center}

\textbf{Case 1: $B$ contains $\{1\}$}

There is a single 1 in each column, so applying $f$ gives a tableau $T^{-}$ of shape $\rho_{n-1}/(k)$, as illustrated below. 

Notice that if a tableau $T^{-}$ of shape $\rho_{n-1}/(1^k)$ has lattice reverse reading word, then it necessarily corresponds to exactly one tableau $T$ of shape $\rho_n/(1^k)$ with lattice reverse reading word where $B$ contains $\{1\}$ (given by adding 1 to each number and adding a box containing 1 to the top of each column). Then these two sets are in bijection, so there are $\alpha_{\rho_{n-1}/(1^k), \nu^{-}}$ tableaux $T$ of shape $\rho_n/(1^k)$ with lattice reverse reading word where $B$ contains $\{1\}$.

Then this case contributes $\alpha_{\rho_{n-1}/(1^k), \nu^{-}}$ tableaux. 
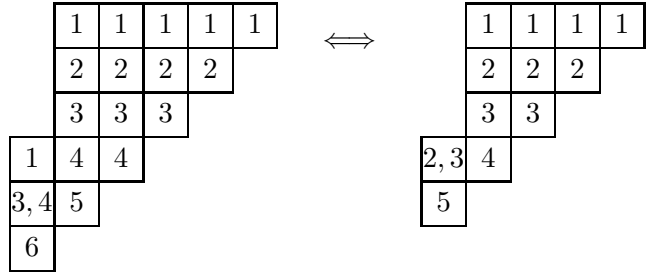
\[
\ytableausetup{notabloids}
\begin{ytableau}
    \none & 1 & 1 & 1 & 1 & 1\\
    \none & 2 & 2 & 2 & 2\\
    \none & 3 & 3 & 3\\
    1 & 4 & 4\\
    3, 4 & 5\\
    6
     \end{ytableau}
      \quad \iff \quad
    \begin{ytableau}
    \none & 1 & 1 & 1 & 1\\
    \none & 2 & 2 & 2\\
    \none & 3 & 3\\
     2, 3 & 4\\
     5
     \end{ytableau}
     \]
\bigskip

\textbf{Case 2: $B$ does not contain a 1}

Then the only 1's in the entire set-valued tableau are in the first row, so $f$ takes $T$ to a tableau of shape $\rho_{n-1}/(k-1)$, as illustrated below. Notice that this is a bijection, so there are $\alpha_{\rho_{n-1}/(1^{k-1}), \nu^{-}}$ tableaux $T$ where $B$ does not contain a 1. 
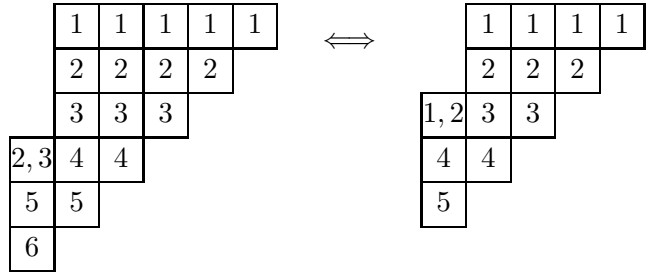
\[
    \ytableausetup{notabloids}
    \begin{ytableau}
    \none & 1 & 1 & 1 & 1 & 1\\
    \none & 2 & 2 & 2 & 2\\
    \none & 3 & 3 & 3\\
    2, 3 & 4 & 4\\
    5 & 5\\
    6
     \end{ytableau}
     \quad \iff \quad 
         \ytableausetup{notabloids}
    \begin{ytableau}
    \none & 1 & 1 & 1 & 1\\
    \none & 2 & 2 & 2\\
    1, 2 & 3 & 3\\
     4 & 4\\
     5
     \end{ytableau}
     \]
\bigskip

\textbf{Case 3: $B$ contains a set of size $\ge 2$ with a 1} 

The 1's in the tableau lie either in the top row or in $B$. Since the box containing $B$ is not deleted by $f$, $T$ is mapped to a tableau of shape $\rho_{n-1}/(k-1)$, as illustrated below. So there are $\alpha_{\rho_{n-1}/(1^{k-1}), \nu^{-}}$ tableaux for this case, since $f$ provides a bijection.
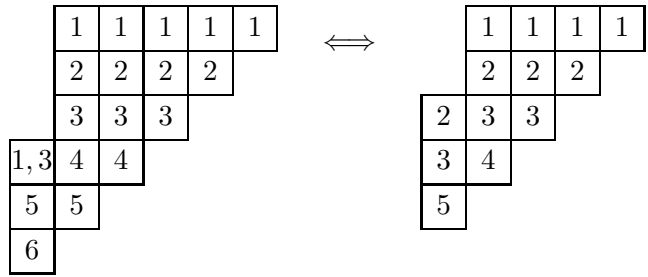
\[
    \ytableausetup{notabloids}
    \begin{ytableau}
    \none & 1 & 1 & 1 & 1 & 1\\
    \none & 2 & 2 & 2 & 2\\
    \none & 3 & 3 & 3\\
    1, 3 & 4 & 4\\
    5 & 5\\
     6
     \end{ytableau}
     \quad \iff \quad 
         \ytableausetup{notabloids}
    \begin{ytableau}
    \none & 1 & 1 & 1 & 1\\
    \none & 2 & 2 & 2\\
    2 & 3 & 3\\
     3 & 4\\
     5
     \end{ytableau}
     \]

Combining all three cases,
\[\alpha_{\rho_n / (1^k), \nu} = \alpha_{\rho_{n-1} / (1^k), \nu^{-}} + 2\alpha_{\rho_{n-1} / (1^{k-1}), \nu^{-}}.\]

\vspace*{-2\baselineskip}
\end{proof}

Using induction and combining the two previous lemmas, we have the following equality between $\alpha$ coefficients.

\begin{lemma}\label{alpha equality}
We have $\alpha_{\rho_{n}/(k), \nu} = \alpha_{\rho_{n}/(1^k), \nu}$ for all positive integers $n$ and nonnegative integers $k$.
\end{lemma}
\begin{proof}
We use induction on $n$ and only consider $k\le n$, because otherwise $\alpha_{\rho_{n}/(k), \nu} = \alpha_{\rho_{n}/(1^k), \nu} = 0$. The base case of $n = 1$ is true because $\rho_1/(k) = \rho_1/(1^k)$ for $k = 0, 1$. Now suppose that for a given $n$, $\alpha_{\rho_{n}/(k), \nu} = \alpha_{\rho_{n}/1^k, \nu}$ for all $k\le n$. Note that $\rho_{n+1}/(k) = \rho_{n+1} = \rho_{n+1}/(1^k)$ for $k = 0$. For $1\le k \leq n$, 

\begin{align*}
\alpha_{\rho_{n+1}/(k), \nu} &= \alpha_{\rho_{n}/(k), \nu^{-}} + 2\alpha_{\rho_{n}/(k-1), \nu^{-}} \\
                        &= \alpha_{\rho_{n}/(1^k), \nu^{-}} + 2\alpha_{\rho_{n}/(1^{k-1}), \nu^{-}} \\
                        &= \alpha_{\rho_{n+1}/(1^k), \nu}.
\end{align*}

In addition, for $k = n+1$, we have $\rho_{n+1}/(n+1) = \rho_{n} = \rho_{n+1}/(1^{n+1})$ by a translation. All together, $\alpha_{\rho_{n+1}/(k), \nu} = \alpha_{\rho_{n+1}/(1^{k}), \nu}$ for all $k \leq n+1$. Thus, by induction, $\alpha_{\rho_{n}/(k), \nu} = \alpha_{\rho_{n}/(1^k), \nu}$ for all positive integers $n$ and nonnegative integers $k$.
\end{proof}
As a result of this relation between the $\alpha$ coefficients, we can now prove a Stembridge-type equality for $G_{\rho/\mu}$ in the special case where $\mu = (k)$ for some nonnegative integer $k$.
\begin{theorem}\label{G k equality}
   There is a Stembridge-type equality for the skew stable Grothendieck polynomial in the case $\mu=(k)$, i.e. $$G_{\rho/(k)} = G_{\rho/(1^k)}.$$
\end{theorem}

\begin{proof}
    Combining Lemma \ref{alpha equality} with Theorem \ref{big G buch},
\begin{align*}
    G_{\rho / (k)} &= \sum_{\nu} (-1)^{|\nu|-|\rho/(k)|}\alpha_{\rho/(k), \nu} G_{\nu} \\
    &= \sum_{\nu} (-1)^{|\nu|-|\rho/(1^k)|}\alpha_{\rho/(1^k), \nu} G_{\nu} \\
    &= G_{\rho / (1^k)}.
\end{align*}

\vspace*{-2\baselineskip}
\end{proof}

\subsection{Proof for All Partitions}

Now, we will use the Hopf algebraic structure to extend this result to all $\mu,$ proving Theorem \ref{main result big g}, an analogue of the Stembridge equality for $G_{\rho/\mu}$.  First, we introduce two definitions and a useful theorem from Buch \cite{buch}.

\begin{definition}
    A \emph{rook strip} is a skew partition $\mu/\sigma$ that contains at most one box in each row and column.
\end{definition}

The following definition (\cite{buch}, Equation 6.4), will allow us to utilize the Hopf algebraic structure of $\Lambda$.
\begin{definition}[Buch]
Define $G_{\lambda//\mu}$ as
\[G_{\lambda // \mu} = \sum_{\sigma} (-1)^{|\mu / \sigma|} G_{\lambda/\sigma},\] 
where the sum is over all $\sigma$ such that $\mu / \sigma$ is a rook strip.
\end{definition}

The polynomials $G_{\lambda/\mu}$ are related to the polynomials $G_{\lambda//\mu}$ by the following theorem, as characterized by Buch (\cite{buch}, Equation 7.4).
\begin{theorem}[Buch]
We have
\[G_{\lambda/\mu} = \sum_{\sigma \subseteq \mu} G_{\lambda // \sigma}.\]
\end{theorem}

From the above definitions and theorem, we may prove the following lemma.
\begin{lemma}\label{big G for k}
If $\rho$ is a staircase shape $\rho = (n, n-1, \dots, 1),$ then $G_{\rho // (k)} = G_{\rho // (1^k)}$. 
\end{lemma}

\begin{proof}

In order for $(k)/\sigma$ to be a rook strip, we need $\sigma = (k)$ or $(k-1)$, so
\[
G_{\rho // (k)} = G_{\rho / (k)} - G_{\rho / (k-1)}.
\]
Similarly, 
\[
G_{\rho // (1^k)} = G_{\rho / (1^k)} - G_{\rho / (1^{k-1})}.
\]
Combining with Theorem \ref{G k equality}, we have $G_{\rho // (k)} = G_{\rho // (1^k)}$. 
\end{proof}

Buch (\cite{buch}, Example 6.8) states that $\Delta(G_{\rho}) = \sum_{\nu \subseteq \rho}G_{\nu} \otimes G_{\rho // \nu}.$
Then, using the skewing operator from Definition \ref{skewing} and the identity $\langle g_{\lambda}, G_{\mu}\rangle = \delta_{\lambda\mu}$, we have
\begin{align*}
g_{\mu}^{\perp}(G_{\rho}) &= \sum_{\nu} \langle g_{\mu}, G_{\nu}\rangle G_{\rho // \nu} \\
&= \langle g_{\mu}, G_{\mu} \rangle G_{\rho // \mu} \\
&= G_{\rho // \mu}.
\end{align*}

In \cite{duality and deformations}, Yeliussizov constructs $\overline{\tau}$ by linearly extending $g_{\lambda}$ $\mapsto g_{\lambda^T},$ and he shows that $\overline{\tau}$ is a ring homomorphism and an involution. We use $\overline{\tau}$ in order to extend Lemma \ref{big G for k} to all $\mu.$ 

\begin{lemma}
Let $\psi$ be an arbitrary ring homomorphism of $\Lambda$. Then the set $A = \{f \in \Lambda: f^{\perp}(a) = \psi(f)^{\perp}(a)$ for $a \in \hat{\Lambda} \}$ is a subalgebra of $\Lambda.$
\end{lemma}

\begin{proof}
From a similar argument as Lemma \ref{subalgebra for arbitrary homomorphism}, using Lemma \ref{skewingProperties2}, we can show that for $f_1, f_2 \in A,$ we have $f_1+f_2 \in A$ and $f_1f_2 \in A.$
\end{proof}

\begin{corollary}\label{closed for big G}
The set $A = \{f \in \Lambda : f^{\perp}(G_{\rho}) = \overline{\tau}(f)^{\perp} (G_{\rho})\}$ is a subalgebra of $\Lambda.$ 
\end{corollary}

\begin{lemma}
We have $G_{\rho // \mu} = G_{\rho // \mu^T}.$
\end{lemma}

\begin{proof}

Let $A = \{f \in \Lambda : f^{\perp}(G_{\rho}) = \overline{\tau}(f)^{\perp}(G_{\rho})\}.$ The polynomials $g_{(k)}$ are elements of $A,$ since \[g_{(k)}^{\perp}(G_{\rho}) = G_{\rho // (k)} = G_{\rho // (1^k)} = g_{(1^k)}^{\perp}(G_{\rho}) = \overline{\tau}(g_{(k)})^{\perp}(G_{\rho}).\]

By definition, \[
g_{(k)} = \sum_P x^P,
\]
summed over reverse plane partitions $P$ of shape $(k)$. For a given reverse plane partition $P,$ the horizontal strip of length $k$ is filled with numbers $i_1 \leq \cdots \leq i_k.$

\begin{center}
    \ytableausetup{notabloids}
    \begin{ytableau}
    i_1 & i_2 & \cdots & i_k
\end{ytableau}
\end{center}

Then, since each number appears once in each column,
\[
g_{(k)} = \sum_P x^P = \sum_{i_1 \leq \cdots \leq i_k} x_{i_1} \cdots x_{i_k}  = h_k.
\]

Now, since $g_{(k)} = h_k,$ we have that $h_k \in A.$ The set $A$ is closed under addition and multiplication by Lemma \ref{closed for big G}, so this means $h_{\lambda} = h_{\lambda_1} h_{\lambda_2} \cdots h_{\lambda_i} \in A.$ Since the $h_{\lambda}$ form a basis for $\Lambda,$ any symmetric function $f = \sum a_{\lambda}h_{\lambda}$ is in $A$ as well. In particular, $g_{\mu} \in A$ for any partition $\mu$. Therefore, 

\begin{align*}
    G_{\rho // \mu} &= g_{\mu}^{\perp}(G_{\rho}) \\
    &= \overline{\tau}(g_{\mu})^{\perp} (G_{\rho}) \\
    &= g_{\mu^T}^{\perp}(G_{\rho}) \\
    &= G_{\rho // \mu^T}.
\end{align*}

\vspace*{-1.5\baselineskip}
\end{proof}

Lastly, we can use these results from the Hopf algebraic structure of $\Lambda$ to prove Theorem \ref{main result big g}.

\begin{proof}

Combining all the above results, we have
\begin{align*}
    G_{\rho / \mu} &= \sum_{\sigma \subseteq \mu} G_{\rho // \sigma} \\
    &= \sum_{\sigma \subseteq \mu} G_{\rho // \sigma^T} \\
     &= \sum_{\sigma^T \subseteq \mu^T} G_{\rho // \sigma^T} \\
     &= G_{\rho / \mu^T}.
\end{align*}

\vspace*{-1.5\baselineskip}
\end{proof}

\section{Acknowledgements}
We would like to thank Adela (YiYu) Zhang for guiding us on our research. We would also like to thank Professor Darij Grinberg for proposing the project and providing helpful suggestions. Finally, we would like to thank the MIT PRIMES-USA program, under which this research was conducted.


\begin{thebibliography}{50}
\bibitem{involve}{Alwaise, E., Chen, S., Clifton, A., Patrias, R., Prasad, R., Shinners, M., Zheng, A. (2017). Coincidences among skew stable and dual stable Grothendieck polynomials. \textit{Involve, a Journal of Mathematics, 11}(1), 143-167.}

\bibitem{buch}{Buch, A. S. (2002). A Littlewood-Richardson rule for the K-theory of Grassmannians. \textit{Acta mathematica, 189}(1), 37-78.}

\bibitem{fomin}{Fomin, S., Kirillov, A. N. (1994). Grothendieck polynomials and the Yang-Baxter equation. In \textit{Proc. Formal Power Series and Alg. Comb} (pp. 183-190).}

\bibitem{galashin}{Galashin, P. (2017). A Littlewood–Richardson rule for dual stable Grothendieck polynomials. \textit{Journal of Combinatorial Theory, Series A}, 151, 23-35.}

\bibitem{hopf}{Grinberg, D., Reiner V. (2020) Hopf Algebras in Combinatorics, version 7. \url{https://arxiv.org/abs/1409.8356}}

\bibitem{LP}{Lam, T., Pylyavskyy, P. (2007). Combinatorial Hopf algebras and K-homology of Grassmanians. \textit{International Mathematics Research Notices}, 2007(9), rnm125-rnm125.}

\bibitem{reiner}{Reiner, V., Shaw, K. M.,  Van Willigenburg, S. (2007). Coincidences among skew Schur polynomials. \textit{Advances in Mathematics}, 216(1), 118-152.}

\bibitem{stanley}{Stanley, R. P. (1990). Enumerative Combinatorics II. \textit{Cambridge Studies in Advanced Mathematics}, 49.}


\bibitem{duality and deformations}{Yeliussizov, D. (2017). Duality and deformations of stable Grothendieck polynomials. \textit{Journal of Algebraic Combinatorics}, 45(1), 295-344.}


\end{thebibliography}
\end{document}